\documentclass[article, 12pt]{amsart}

\usepackage{amssymb,latexsym,amsfonts, amsthm, amsmath, amsrefs}
\usepackage{enumerate}
\usepackage[shortlabels]{enumitem}
\usepackage{array}
\usepackage{longtable}
\usepackage{multicol}
\usepackage{multirow}
\newcolumntype{$}{>{\global\let\currentrowstyle\relax}}
\newcolumntype{^}{>{\currentrowstyle}}
\usepackage{booktabs}

\usepackage{fancyhdr}
\usepackage{verbatim}
\usepackage{geometry}
\usepackage{fullpage}
 \usepackage{tikz}
\usepackage[all,cmtip]{xy}
\usetikzlibrary{shapes,arrows,chains,matrix,positioning,scopes}
\makeatletter
\tikzset{join/.code=\tikzset{after node path={%
\ifx\tikzchainprevious\pgfutil@empty\else(\tikzchainprevious)%
edge[every join]#1(\tikzchaincurrent)\fi}}}
\tikzset{>=stealth',every on chain/.append style={join},
         every join/.style={->}}

\input xy
\xyoption{all}

\newtheorem{thm}{Theorem}[section]   

\newtheorem{cor}[thm]{Corollary}     
\newtheorem{lem}[thm]{Lemma}         
\newtheorem{prop}[thm]{Proposition}  

\theoremstyle{definition} 
\newtheorem{defn}[thm]{Definition}   
\newtheorem{remark}{Remark}   
\newtheorem{ex}[thm]{Example}        

\newtheorem{prob}[thm]{Problem}

\newcommand{\mc}{\ensuremath{\mathcal}}

\newcommand{\ZZ}{\ensuremath{\mathbb{Z}}}
\newcommand{\RR}{\ensuremath{\mathbb{R}}}

\newcommand{\FF}{\ensuremath{\mathbb{F}}}

\usepackage{parskip}
\usepackage{color}
\usepackage{soul}

\title{Convex Neural Codes in Dimension 1}

\author{Zvi Rosen, University of Pennsylvania, PA \\
Yan X Zhang, San Jose State University, CA}

\begin{document}
\setlength{\parskip}{5pt}

\begin{abstract}
 {\em Neural codes} are collections of binary strings motivated by 
patterns of neural activity.
In this paper, we study algorithmic and enumerative aspects of convex
neural codes in dimension $1$ (i.e. on a line or a circle). 
We use the theory of {\em consecutive-ones matrices} to obtain 
some structural and algorithmic results; we use generating 
functions to obtain enumerative results.


\end{abstract}

\maketitle

\section{Introduction}
\label{sec:introduction}
\subsection{Motivation}

The 2014 Nobel Prize in Medicine was awarded in part
for the discovery of neurons called  \emph{place cells} \cites{rat1,rat2,moser}.
In experiments, each place cell fires when a rat occupies
a specific region in its physical environment; 
the regions that trigger this response are
called \emph{place fields}. Experimental data indicates that place fields
are roughly convex. In practice, place fields are not strictly convex; 
still, convexity is useful as an abstraction 
to describe the essence of the empirical phenomena. 
Binary strings called \emph{codewords} capture the 
intersection patterns of convex sets in $\RR^d$ 
detected by ``sensors'' in space; we
call the set of all detected codewords 
the \emph{(convex neural) code}.  Seen in this light,
neural activity patterns are translated into algebraic-combinatorial
objects, and a whole new toolbox becomes relevant to their analysis.
 This research program was initiated quite recently 
\cite{curto1,curtoCoding}, and has already inspired a large body of
research, e.g. \cites{Curtoetal,itskov,shiu,itskov2016,gky2016, groebner,omar}. Our manuscript aims to supplement this 
literature with the 1-d case, using tools from the study of consecutive-ones matrices and algebraic combinatorics.

\begin{ex} The top image in Figure \ref{fig:pfex} describes the
arrangement of place fields in the rat's environment. The lower image 
is the corresponding interval arrangement in $\RR^1$. The binary
strings below indicate the corresponding codewords.
\end{ex}

\begin{figure}
\includegraphics[height = 2.7in]{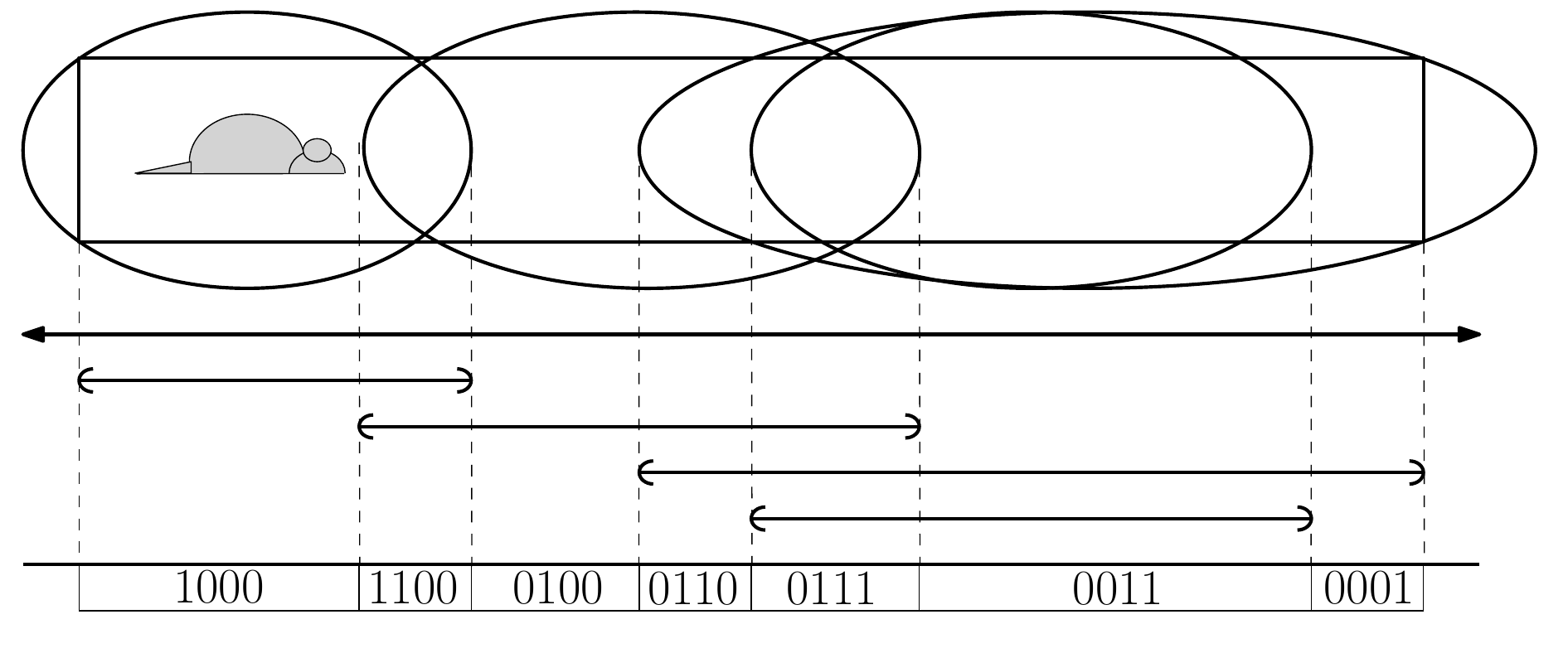}
 \caption{A rat on a linear track with $4$ place fields (convex regions) detected by neurons. For each area (space between dotted lines), a sensor in the area would give the corresponding codeword on the bottom when detecting which regions cover it. \label{fig:pfex}}
\end{figure}

The $1$-dimensional case ($1$-d) is a natural stepping stone to higher dimensions. In this work, we explore the combinatorial properties of neural codes in $1$ dimension from various angles such as consecutive-ones matrices and generating functions. 

\subsection{Mathematical Setup}

Consider some ambient space $X$ (usually $X = \RR^d$, but we are also interested in $S^1$ for this paper). Let $\mc{U} = \{U_1,\ldots,U_n\}$ be a set of open convex subsets of $X$. For a point $p \in X$, let the corresponding \emph{codeword} be the image of $p$ under the map $c: X \to \mathbb{F}_2^n$, where:
\[ c(p) = (c_1(p),\ldots, c_n(p)), \hspace{1cm} \text{ where    }
c_i(p) = \begin{cases} 1 & p \in U_i. \\ 0 & p \notin U_i. \end{cases} \] 
For a set of \emph{sensors} $S \subseteq X$ (intuitively, one can imagine a sensor as a device placed at a point which can detect if each place field covers it), we call the image of $S$ under $c()$ the \emph{convex neural code} $\mc{C}(\mc{U}, S)$. We call convex neural codes simply \emph{codes} for the remainder of this paper.

We say that a code $\mathcal{C}$ is \emph{realizable in dimension $d$} if $X$ is locally $d$-dimensional and there exists $\mc{U}$ with $U_i \subseteq X$ such that $\mc{C} =
\mc{C}(\mc{U}, X)$ (equivalently, all possible codewords are detected) in which case we say $\mc{U}$ \emph{realizes} $\mc{C}$. There are various partial results for the dimension
of realizability of codes; one of the authors' joint work \cite{Curtoetal} shows that if a code contains the all-$1$'s codeword, it is realizable in $2$-d.

In this work, we tackle the following questions:

\begin{prob}[Reconstruction] 
Given a code $\mc{C}$, detect if there exists a set of convex regions realizing the code, and try to construct one if there is.
\label{prob:recon}
\end{prob}

\begin{prob}[Algebraic Signatures] 
Find an algebraic-combinatorial signature of the code $\mc{C}$ that
indicates whether or not it is realizable under specific parameters.
\label{prob:sig}
\end{prob}

\begin{prob}[Enumeration]
With $n$ sensors, classify/enumerate the possible codes, multiset of codes, etc.
\label{prob:enum}
\end{prob}

We already see nontrivial mathematics in $1$-d, even though convex regions are simply intervals along a line. We call a set of intervals in $\RR$ (or $S^1$) an \emph{interval arrangement}. We now discuss some nuances in our setup. 

\begin{enumerate}
\item First, we vary the \textbf{topological constraints}. 
It is easy to get carried away by minutiae, but different 
real-world scenarios may warrant different constraints. 
For example, we may have so many sensors in $S$ as to guarantee that 
we would pick up all possible codewords for an arrangement. On the other end, we may only have a few sensors in our sensor set $S$ and miss potential codewords. We refer to these 
regimes as being {\em sensor-dense} (i.e. $\mc{C}(\mc{U},S) 
= \mc{C}(\mc{U},\RR)$) and {\em sensor-sparse} (i.e. $S$ is any finite 
set of points) respectively.  

A different line of inquiry about $1$-d topology 
is the following: suppose we pick up all possible codewords 
in an arrangement. What would happen if place fields must be
open convex intervals? What would happen 
if place fields can be convex intervals with any boundary type? 
The first directly corresponds to our sensor-dense regime. 
Interestingly, the second is equivalent to the sensor-sparse 
regime for $1$-d, despite the question being differently motivated. 
We explore the precise nature of the equivalence in Appendix~\ref{sec:topology}.


\item Second, we can vary the \textbf{geometry} of our ambient set. The image in Figure \ref{fig:pfex} depicts a linear track. However, one could translate the setup to a circular
track, where each $U_i$ is a segment of $S^1$ homeomorphic to an interval. Strictly speaking, most subsets of $S^1$ (including $S^1$ itself!) are not convex as subsets of Euclidean space. However, arcs on $S^1$ are a reasonable analogue of intervals in $\RR^1$ as ``convex'' sets. One can see elsewhere, e.g. \cite{horn1949}, for similar treatments. From an applied perspective, we were inspired to include the circular case by \cite{hollup2001accumulation}, where rats running around a circular track have place fields corresponding to arcs on the track. We say we are in the \emph{linear} case if we have $k$ intervals and $n$ sensors on a line and in the \emph{circular} case if we have $k$ arcs and $n$ sensors on a circle.
\item Finally, the reconstruction problem can be also stated for \textbf{multisets}. In the definition of the code, we considered the set of codewords without multiplicity. When we take a finite number of sensors (points at which the codeword function is evaluated), the result is a multiset of codewords, which contains additional information for consideration.
\end{enumerate}

Problem~\ref{prob:recon} may be of interest to experimentalists doing a posteriori analysis using the code set as an invariant; we discuss reconstruction in Sections~\ref{sec:sparse-reconstruction}, \ref{sec:dense-reconstruction}, and \ref{sec:multiset-reconstruction}. Problem~\ref{prob:sig} may help resesarchers looking to reject faulty codes; we give a result for the sensor-sparse variation in Section~\ref{sec:sparse-reconstruction}. We study Problem~\ref{prob:enum} in Section~\ref{sec:linear-enumeration} and \ref{sec:circular-enumeration} obtaining some asymptotic information and interesting bijections involving subspaces of $\FF_2^{n+1}$ with full support and labeled graded $(3)$-avoiding posets. In Section~\ref{sec:conclusion}, we give some general remarks for future work.

\subsection{Main Results}

Here is an illuminating idea, first pointed out by Anne Shiu:

\begin{prop} Codes realizable in dimension
$1$ are exponentially small relative to the full set of codes on $n$ neurons.
\end{prop}

\begin{proof}[Proof Idea]
Suppose the arrangement of intervals $\{I_\alpha\}$ realizes
the code $\mathcal{C}$. We read the appearing codewords from left to right; each new codeword appears only when an interval
appears or when an interval disappears. This implies that there are a
maximum of $2n + 1$ distinct codewords in $\mathcal{C}$, where
$n$ is the number of neurons, i.e. the length of each codeword.
This is an exponentially tiny proportion of codes as $n$ grows.
\end{proof}

The insight that led to this observation turns out to be powerful -- in particular, a
realization of a code $\mathcal{C}$ amounts to an
ordering of the codewords satisfying a discrete version of convexity, 
namely that the $1$'s must occur contiguously in the ordering for each neuron. To be precise, let the \emph{sensor matrix} of $k$ intervals and $n$ sensors in the linear case be a $k \times n$ matrix $M$ where $M_{i,j} = 1$ if the $j$-th sensor detects the $i$-th interval, and $M_{i,j} = 0$ otherwise (note that the set of column vectors in $M$ is just $\mathcal{C}$). 

Call a length $n$ $0$-$1$-vector with a single contiguous block of $1$'s a \emph{discrete interval}. Note that the row vectors of the sensor matrix $M$ are discrete intervals; for example, reading the second coordinate in Figure~\ref{fig:pfex} gives the discrete interval $0111100$. In the circular case, we can think of the sensor matrix $M$ as an equivalence class of $k \times n$ $0$-$1$ matrices where two matrices are equivalent up to cyclically permuting the columns, and each row becomes a discrete interval under some rotation.

This combinatorial characterization connects the question
of $1$-d realizable codes to the well-studied problem
of {\em consecutive-ones matrices}, and we can use its technology (namely, $PQ$- and $PC$-trees) to tackle the reconstruction problem for most of our cases; sensor-dense reconstruction in $S^1$ is subtly difficult. 

\begin{thm}[Sensor-sparse]
Any code realizable in $\RR^1$ can be transformed into a sensor-sparse linear
realization via the PQ-tree algorithm.
Any code realizable in $S^1$ can be transformed into a sensor-sparse linear
realization via the PC-tree algorithm.
\end{thm}
 In the sensor-dense version, we need
an openness condition not required by general consecutive-ones
matrices, so the next connection requires some work to prove:
\begin{thm}[Sensor-dense]
Any code realizable in $\RR^1$ can be transformed into a sensor-dense linear
realization via the PQ-tree algorithm.
\end{thm}
The corresponding statement for realizations in $S^1$ is not obviously
true. We present an example to demonstrate this difficulty later.
Regarding the multiset variation on the problem, we find 
the following result:

\begin{thm}
Let $\mc{C}$ be a code, and $m: \mc{C} \to \mathbb{N}$ be a
multiplicity for the appearance of each codeword.
In the sensor-sparse regime, on the line and the
circle, if there exists $\{U_i\}$ realizing the code, it is also
possible to realize the code with the desired multiplicity. In the
sensor-dense case, if a code is realizable, then it is realizable with
any multiplicity greater than or equal to some minimal multiplicity $m'$.
\end{thm}

For the algebraic signature problem, we found a nice signature 
for the linear sensor-sparse case from the consecutive-ones literature:

\begin{prop}
A code $\mc{C}$ is sensor-sparse realizable in $\RR^1$ if and only if an
associated graph $I(\mc{C})$ is bipartite.
\end{prop}

Finally, we obtain some enumerative results in Table~\ref{table:enum}; instead of enumerating codes directly, we count a closely related object called \emph{discrete interval sets}, which still give a grasp on the number of such codes but are easier to count. 

\begin{table}[h!]
  \centering
\setlength{\extrarowheight}{5pt}
  \begin{tabular}{l|l|r}
    Topological Regime & Geometry & Number of discrete interval sets \\ \hline
    Sparse & Line & $\binom{\binom{n+1}{2}}{k}$ \\
    Sparse & Circle & $\binom{n^2 - n + 1}{k}$ \\
    Dense & Line &  Coefficient of $x^ny^k$ in $\sum_{m=0}^\infty \frac{x^m}{(1-a_1x)(1-a_2x)\cdots(1-a_{m+1}x)}$\\
    Dense & Circle & Coefficient of $x^ny^k$ in $1 +
    \sum_{m=1}^\infty \frac{2x^m}{(1-a_mx)^{m+1}}$ \\
  \end{tabular}
\vspace{5mm}

\caption{Our enumerative results. We define $a_i = (1 + y)^i - 1$ for all $i$ \label{table:enum}.}
\end{table}

\section{The Consecutive-Ones Property}
\label{sec:1d-linear}

Recall that we have two regimes: sensor-sparse and sensor-dense. We now translate these intuitive notions into formalism, starting with the linear case, where we borrow terminology from \cite{consecutiveones}. 
We say that a $0$-$1$ matrix $M$ has the \emph{consecutive-ones property} (equivalently, we say \emph{$M$ is CO}) if each row is a discrete interval. These matrices will correspond to the sensor-sparse regime. Luckily for us, the problems in this regime correspond naturally to the problems that come up in the consecutive-ones property literature. 

\begin{minipage}[t]{.8\textwidth}
We now add a twist to our definition. We say that a $0$-$1$ matrix has two \emph{inharmonious} columns $x$ and $y$ if, with some two rows, they induce one of the $2 \times 2$ submatrices $(a)$ or $(b)$ on the right. If such a matrix has no pair of inharmonious columns, we say the matrix is \emph{harmonious}. We remark that $(x,y)$ is harmonious if and only if $x,y$ are comparable in the boolean lattice structure on $\{0,1\}^k$. We say that a $0$-$1$ $CO$ matrix $M$ is \emph{inharmonious} if some two adjacent columns are inharmonious. Otherwise, we say $M$ is \emph{harmonious} (equivalently, we say \emph{$M$ is $HCO$}). These matrices will correspond to the sensor-dense regime. 
\end{minipage} \hspace{2mm}
\begin{minipage}[t]{.2\textwidth}
\vspace{0mm}

\begin{tabular}{cc}
$\begin{bmatrix}
0 & 1 \\
1 & 0
\end{bmatrix}$ &
 $ \begin{bmatrix}
1 & 0 \\
0 & 1
\end{bmatrix}$ \\[6mm]
(a) & (b)
\end{tabular}
\end{minipage}

For the $1$-dimensional circular case, we extend the notion of \emph{CO} to the circular analogue \emph{circular consecutive-ones (CCO)} where we allow the discrete interval of $1$'s to ``wrap around,'' i.e. counting going from the last column to the first column as contiguous.  Equivalently, we allow each row to have only one contiguous block of $1$'s \textbf{or} only one contiguous block of $0$'s. As an example, the following matrix is CCO but not CO:
\[
\begin{bmatrix} 1 & 1 & 0 & 0 & 1 \\
0 & 1 & 1 & 0 & 0 \\ 
0 & 0 & 1 & 1 & 1 \end{bmatrix}
\]

As in the linear case we have sensor-sparse and sensor-dense regimes, corresponding to CCO matrices with no further restriction and CCO matrices with no two consecutive (here the first and last column vectors are also consecutive) inharmonious columns, which we call \emph{HCCO}. 

We now make the correspondence between the  regimes and these matrices explicit. 

\begin{prop} 
\label{prop:matrices-realizability}
The following are true of a code $\mc{C}$:

\begin{tabular}{lll}
1. $\mc{C}$ is&sensor-sparse realizable in&$\RR^1$ if and only if it is the column set of a CO matrix. \\ 
2. $\mc{C}$ is&sensor-dense realizable in&$\RR^1$ if and only if it is the column set of an HCO matrix. \\
3. $\mc{C}$ is&sensor-sparse realizable in&$S^1$ if and only if it is the column set of a CCO matrix. \\
4. $\mc{C}$ is&sensor-dense realizable in&$S^1$ if and only if it is the column set of an HCCO matrix. \\
\end{tabular}
\end{prop}
\begin{proof}
For the ``only if'' direction, it is sufficient to provide 
an ordering of the codewords so
that the resulting matrix has the desired property. 
Begin with a realization of the code as an arrangment 
of open intervals in $\RR^1$ or $S^1$. In $\RR^1$, begin at $-\infty$ and
read codewords at the set of sensors, sorted from left to right. 
In $S^1$ begin at any sensor $s$ on the circle, 
and go in either direction (say clockwise), stopping 
before the original point $s$ is reached.

The (possibly circular) consecutive-ones property is guaranteed by convexity of the code: once a sensor detects an interval (corresponding to a $1$) and stops detecting the interval (corresponding to a $0$), it can no longer detect the interval again. This precisely corresponds to having a discrete interval of $1$'s in the corresponding row. 

In the sensor-dense regime, suppose an inharmonious pair of adjacent columns (corresponding to sensors at $k$ followed by $l$) exist in the corresponding matrix for rows $i$ and $j$. Without loss of generality, this means that there are some two open intervals (corresponding to $i$ and $j$) that at two sensors $k$, $i$ is detected but not $j$, and at $l$, we stop detecting $i$ but detect $j$. However, two open intervals on $\RR$ or $S^1$ must either intersect or have a point strictly between them. So either a sensor between them would have picked up both intervals or neither, contradicting the fact that these columns were adjacent.

To go in the reverse direction, pick an $0.5 > \epsilon > 0$ and create a sensor for each column of the matrix spaced $1$ unit apart. Each row of the matrix corresponds to some discrete interval $i, i+1, \ldots, j$ (possibly wrapping around for the circuluar case), except for two degenerate cases where the entire row is $0$ or $1$. For the non-degenerate cases, create an open interval $(i-\epsilon, j+\epsilon)$. For the degenerate cases, create the empty set and the entire set ($\RR$ or $S^1$) for the $0$- and $1$-row cases respectively. It is easy to check the resulting arrangement satisfies the definitions.
\end{proof}

\section{The Reconstruction and Signature Problems}

Using the language of consecutive-ones / CO and Proposition~\ref{prop:matrices-realizability}, the reconstruction problem is the following: given a set of $n$ codewords $L$, decide if there exists a CO matrix (with at least $n$ columns) such that the set of columns is $L$, and if so, try to produce the matrix. The sensor-sparse regime (for both linear and circular cases) is quickly understood after translating to the language of the consecutive-ones property. We also supply a solution for the linear (but not circular) case of the sensor-dense regime.

\subsection{Sensor-Sparse Regime}
\label{sec:sparse-reconstruction}

The following result almost solves the sensor-sparse regime problem completely for both the linear and circular cases:

\begin{prop}
\label{prop:soft set reconstruction}
Given a set of $n$ distinct length-$k$ vectors $L$, there is a $O(n+k)$ algorithm that decides if there exists a CO (or CCO) matrix with exactly $n$ columns and column set equal to $L$ and constructs one if it exists. 
\end{prop}
\begin{proof}
The algorithm uses a data type known as a \emph{PQ-tree} for the linear situation; for the circular situation the analogue is the \emph{PC-tree}. 
For a proof, see \cite{hsu-circular}. We will include an 
example in Section~\ref{sec:pctree}.
\end{proof}

\begin{remark}
As noted by Hsu and McConnell \cite{hsu-circular},
 much of the work on the consecutive-ones property involves an 
unnecessary symmetry-breaking of the circle.
Indeed, the consecutive-ones property problems are usually 
subsumed by their circular counterparts and the structures 
obtained in the circular case (\emph{PC-trees} as opposed 
to the linear \emph{PQ-trees}) are mathematically cleaner.
\end{remark}

\begin{cor}
\label{cor:soft set reconstruction}
Given a set of $n$ distinct length-$k$ vectors $L$, there is a $O(n+k)$ algorithm that decides if there exists a CO (or CCO) matrix with at least $n$ columns such that the column set is equal to $L$ and constructs one if it exists. 
\end{cor}
\begin{proof}
Since removing any column from a CO matrix preserves the CO property, a CO matrix with column set equal to $L$ gives a CO matrix with $n$ columns (where each column appears exactly once) by removing duplicates, which still has the column set equal to $L$. Proposition~\ref{prop:soft set reconstruction} then finishes the problem. The logic also holds for the CCO case.
\end{proof}

We can also import an algorithm for code rejection 
from the consecutive-ones literature for realizability of
a sensor-sparse code in $\RR^1$.

\begin{defn}[{\cite[Definition~6.1]{mcconnell04}}]
Let $M$ be a binary matrix. The incompatibility graph $I(M)$
is the graph defined by:
\begin{itemize}
\item[$V =$]  ordered pairs of columns $(a,b)$ signifying ``$a$ is to
the left of $b$''.
\item[$E =$] incompatible pairs of relations, of the form:
\begin{enumerate}
\item  $\{(a,b),(b,a)\}$, and
\item  $\{(a,b), (b,c)\}$, where there is a row in which $a$ and $c$
have $1$ while $b$ has $0$.
\end{enumerate}
\end{itemize}
\end{defn}

If a consecutive-ones ordering of $M$ exists, then there is a set
of $\binom{n}{2}$ relations that have no mutual incompatibility;
i.e. the induced subgraph of $I(M)$ is empty. McConnell proves that
the converse is true as well:

\begin{thm}[{\cite[Theorem 6.1]{mcconnell04}}]
$M$ has a CO ordering if and only if $I(M)$ is bipartite.
For a code $\mc{C}$, this means $\mc{C}$ has a sensor-sparse
realization in $\RR^1$ if and only if $I(\mc{C})$ is bipartite.
\end{thm}

\begin{ex} Consider the set of codewords given by the columns
of the matrix \\ below right.

\vspace{-4mm}

\begin{minipage}{.6\textwidth}
Using the indicated column labels, we have the
following odd cycle in $I(M)$:
$\{(d,a),(a,b),$ $(b,c),(c,a),(a,c)\}$. This is
sufficient to demonstrate that $I_C(M)$ is not 
realizable in 1-D.
\end{minipage} \begin{minipage}{.4\textwidth}
\centerline{
\bordermatrix{
& a & b & c & d \cr
& 1 & 1 & 0 & 1 \cr
& 1 & 0 & 1 & 1 \cr
& 0 & 1 & 0 & 1 \cr
& 0 & 0 & 1 & 1 \cr}
}
\end{minipage}
\end{ex}

\subsection{Sensor-Dense Regime}
\label{sec:dense-reconstruction}

In this section, we extend the results from the sensor-sparse
regime to the sensor-dense regime. First we make some 
definitions.

\begin{defn}
Let $L$ be a set of $n$ distinct length-$k$ column vectors.
An \emph{ordering} of $L$ is a $k \times n$ matrix 
where the columns form a permutation of $L$.
A \emph{multiordering} of $L$ is a $k \times n'$ matrix
 where $n' \geq n$ and every column vector appears at least once.
\end{defn}
 One sign that the sensor-dense regime is more difficult 
is that there exist sets $L$ for which there exist 
harmonious CO multiorderings but no harmonious CO orderings; 
in other words, the reconstruction would need to 
``know'' when to  duplicate certain vectors. 

\begin{ex}
Consider the set of vectors below left. No two of the first 
three column vectors can be adjacent in a harmonious CO ordering, 
though we can use the fourth vector to ``pad'' them, 
producing the harmonious CO multiordering at right. \[
\left\{\begin{bmatrix} 1 \\ 0 \\ 0 \end{bmatrix},
\begin{bmatrix} 0 \\ 1 \\ 0 \end{bmatrix},
\begin{bmatrix} 0 \\ 0 \\ 1 \end{bmatrix},
\begin{bmatrix} 0 \\ 0 \\ 0 \end{bmatrix}\right\}
\hspace{2cm}
\begin{bmatrix} 1 & 0 & 0 & 0 & 0 \\
0 & 0 & 1 & 0 & 0 \\ 
0 & 0 & 0 & 0 & 1 \end{bmatrix}.\] 
\label{ex:padding}
\end{ex}

Clearly, any HCO multiordering can be trimmed down to a CO ordering 
by removing duplicates, as removing vectors cannot break 
the CO condition. Thus, a naive approach to the set reconstruction 
problem, inspired by Proposition~\ref{prop:soft set reconstruction} 
and Example~\ref{ex:padding}, is to obtain a CO ordering, 
then extend it to an HCO multiordering. 

However, a very subtle concern is that even if $L$ actually 
has some HCO multiordering, this algorithm may start with a 
``bad'' CO ordering that is not extendable to an HCO multiordering. What we would need is something like the following claim, which, luckily for us, is true:

\begin{prop}
\label{prop:extension}
Suppose $L$ has at least one HCO multiordering. Then any CO ordering of $L$ can also be extended into an HCO multiordering.
\end{prop}

The remainder of this section is a detailed proof of this
assertion. First, we fix some notation and definitions.
\begin{defn}
\label{defn:connected}
Let $R$ be the set of rows in a sensor matrix.
\begin{itemize}[leftmargin=5mm]
\item For a set of rows $S \subset R$, 
we say $x \in L$ \emph{is $0$ (resp. $1$) on $S$} if for all rows in 
$S$, $x$ has a $0$ (resp. $1$) in that row.

\item For any $S \subset R$, define an undirected graph 
$G_S$ with vertex set $L$ and edge set 
\[ E = \{ (x,y)\: \vline\: \: \exists r \in S 
\text{ such that } x \text{ and } y \text{ are both 1 on } r \}.\] In other words, $G_s$ detects if two sensors ``overlap'' on something in $S$.

\item Let $x,y \in L$ be \emph{$S$-connected} if they 
 are in the same connected component of 
$G_S$.

\item Let $x,y \in L$ and $r_1,r_2 \in R$ be such that the restriction
of $x, y$ to the rows $r_1,r_2$ is inharmonious. If $z \in L$ is
$1$ on $\{r_1,r_2\}$ then $z$ is \emph{bound between 
$x$ and $y$}. 
\end{itemize}
\end{defn}

In the following lemma, we summarize the
 main structural tools we use to explore our multiorderings.

\begin{lem}
\label{lem:row selection}
The following hold for any CO multiordering of $L$:
\begin{enumerate}[leftmargin=5mm]
\item If $x$ and $y \in L$ are $S$-connected and $z \in L$ is $0$ on $S$, then $z$ cannot be between $x$ and $y$.
\item If $z \in L$ is bound between $x \in L$ and $y \in L$, then in any $CO$ ordering of $L$, $z$ must be between $x$ and $y$. 
\end{enumerate}
\end{lem}
\begin{proof}
To see (1), note that $z$ would otherwise violate convexity of the intervals labeled by $S$. For (2), suppose $x$ were to the left of $y$. If $z$ were to the right of $y$, $y$ would violate the convexity of the interval containing both $x$ and $z$. By symmetry, $z$ cannot be to the left of $x$. Thus, $z$ must be between them.
\end{proof}

\begin{proof}[Proof of Proposition~\ref{prop:extension}]
Suppose $L$ has the HCO multiordering $A$. 
Consider any CO ordering $B$ of $L$. 
We claim we can insert columns between adjacent 
inharmonious columns of $B$ to create an HCO multiordering. The idea is to ``copy and paste'' them from $A$.

Take any inharmonious pair of columns $x$ and $y$. In $A$, pick any pair of 
columns equal to $x$ and $y$ and call them $x$ and $y$ 
as a slight abuse of notation. Without loss of generality, suppose $x$ is left of $y$. 
Consider the submatrix $M$ formed from these two columns 
and all columns between them. 
$x$ and $y$ partition the rows into four sets $R_{ij}$, 
where $i, j \in \{0,1\}$ and $R_{ij}$ contains rows $r$ 
where $M_{r, x} = i$ and $M_{r, y} = j$. 
Note that $R_{10}$ and $R_{01}$ are nontrivial 
since $x$ and $y$ are inharmonious. 
Also, every vector in $M$ is $1$ on $R_{11}$ (which could be trivial).


{\bf Claim:} {\em There is at least one vector $v$ in $A$ between $x$ and $y$
which is $0$ on $R_{10} \cup R_{01}$.} \\
By Lemma \ref{lem:row selection}, since $x$ and $y$ 
are adjacent in $B$, which is CO, there must not be 
any vectors in $L$ bound between $x$ and $y$.

To see this, consider the rightmost vector $x'$ 
which is $0$ on $R_{01}$ and the leftmost vector $y'$ 
which is $0$ on $R_{10}$. If $x' = y'$, this is the desired $v$.
If $x'$ is to the right of $y'$, both vectors must be $0$ on
$R_{10} \cup R_{01}$, by convexity.

Suppose $x'$ is strictly to the left of $y'$. By assumption,
$x'$ is not completely $0$ on $R_{10}$ while $y'$ is not completely 
$0$ on $R_{01}$. Note $x'$ cannot be adjacent to $y'$, 
since this would mean $x'$ and $y'$ are not harmonious. 
Thus, there is some vector $z$ strictly between 
them that has value $1$ on some row $r_1 \in R_{10}$ 
and some $r_2 \in R_{01}$. But this means $z$ is bound between 
$x$ and $y$, which gives a contradiction. 

{\bf Claim:} {\em There is at least one vector $v$ in $A$ between $x$ and $y$
which is $0$ on $R_{10} \cup R_{01} \cup R_{00}$.} \\
Based on the last claim, we have a nontrivial submatrix
$M'$ comprised of the contiguous columns 
that are $0$ on $R_{10} \cup R_{01}$. 
Call its leftmost and rightmost vectors $v_1$ and 
$v_2$ respectively. Suppose, by way of contradiction,
that every column in $M'$, 
including $v_1$ and $v_2$, must have a $1$ somewhere 
on $R_{00}$. Consider the vector $x'$ in $M$ 
directly left of $v_1$; note that:

\begin{itemize}
\item $x'$ must have a $1$ in $R_{10}$ (otherwise it would have been in $M'$).
\item $x'$ must be $0$ on $R_{01}$ (since it is to the left of $M'$).
\item $x'$ must have at least a $1$ in $R_{00}$ (in fact, all the $1$'s that $v_1$ has, since $x' > v_1$ by the fact that $x'$ has a $1$ somewhere in $R_{10}$ and $M$ is harmonious). In particular, $x' \neq x$. 
\end{itemize}

The same argument gives a vector $y' \neq y$ right of $v_2$ in $M$. An example of such a situation is given in Figure~\ref{fig:transition example}.

\begin{figure}[h!]
\setlength\extrarowheight{2pt}
\begin{longtable}[c]{c|ccc|ccc|ccc|}
\multicolumn{1}{c}{} & \multicolumn{9}{c}{$\overbrace{\hspace{6cm}}^{\text{\normalsize $M$}}$} \\
\multicolumn{4}{c}{} & \multicolumn{3}{|c|}{$\overbrace{\hspace{1.6cm}}^{\text{\normalsize $M'$}}$} & \multicolumn{3}{c}{} \\
 \multicolumn{1}{c}{} & $x$ & $\cdots$ & $x'$ & $v_1$ & 
& $v_2$ & $y'$ & $\cdots$ & \multicolumn{1}{c}{$y$} \\[2mm]
\cline{2-10}
$R_{11}$ & 1 & $\cdots$ & 1 & 1 & 1 & 1 & 1 & $\cdots$ & 1 \\
\cline{2-10}
\multirow{2}{*}{$R_{10}$}& 1 & $\cdots$ & 1 & 0 & 0 & 0 & 0 & $\cdots$ & 0 \\
& 1 & $\cdots$ & 0 & 0 & 0 & 0 & 0 & $\cdots$ & 0 \\
\cline{2-10}
\multirow{2}{*}{$R_{01}$} & 0 & $\cdots$ & 0 & 0 & 0 & 0 & 1 & $\cdots$ & 1 \\
& 0 & $\cdots$ & 0 & 0 & 0 & 0 & 0 & $\cdots$ & 1 \\
\cline{2-10}
\multirow{3}{*}{$R_{00}$} & 0 & $\cdots$ & 1 & 0 & 0 & 0 & 0 & $\cdots$ & 0 \\
& 0 & $\cdots$ & 1 & 1 & 1 & 0 & 0 & $\cdots$ & 0 \\
& 0 & $\cdots$ & 0 & 0 & 1 & 1 & 1  & $\cdots$ & 0 \\
\cline{2-10}
\end{longtable}
\caption{An example $M$, $M'$, and $R_{ij}$ illustrating the proof.}
\label{fig:transition example}
\end{figure}

The key observation is that $x'$ is $R_{00}$-connected to $y'$. Because $x' > v_1$ and $y' > v_2$, it suffices to prove that $v_1$ is $R_{00}$-connected to $v_2$. To see this, note that on $M'$ the sets of rows $R_{11}, R_{10}, R_{01}$ are the only places that values may change between $v_1$ and $v_2$. So each step from $v_1$ to $v_2$ consists of adding $1$'s in $R_{00}$ or removing $1$'s but not both. Furthermore, we can never remove all the $1$'s by our condition that all columns in $M'$ must have a $1$ somewhere on $R_{00}$. Therefore, $x'$ is $R_{00}$-connected to $y'$.

Given this information, we claim that it is impossible for $x$ to be adjacent to $y$ in $B$. If this were the case, $y'$ must appear to the right of $y$, because of Lemma~\ref{lem:row selection} and the fact that $y$ and $y'$ have a $1$ shared somewhere in $R_{01}$ whereas $x'$ doesn't. Similarly, $x'$ must appear to the left of $x$. But $x'$ and $y'$ are $R_{00}$-connected and both $x$ and $y$ are $0$ on $R_{00}$, so Lemma~\ref{lem:row selection} gives a contradiction. 

Thus, we must have had a vector $z$ that was $0$ 
everywhere except on $R_{11}$. This means we can 
just insert $z$ between $x$ and $y$ to make this 
part of our multiordering harmonious; furthermore, 
this operation does not affect harmoniousness or 
CO-ness anywhere else. Applying this to all 
adjacent pairs in $B$ gives an HCO multiordering $B'$, as desired.
\end{proof}
We now are able to prove our main result of this section:
\begin{thm}
\label{thm:1d reconstruction CO-harmonious}
Let $L$ be a set of $n$ length-$k$ codewords, in the
 sensor-dense regime for $\RR^1$. There is a 
constructive $O(n+k)$ algorithm that produces an HCO
multiordering of $L$; equivalently, it outputs an interval
arrangement realizing $L$ as a code.
\end{thm}
\begin{proof}
First, if there is an HCO multiordering of $L$, 
removing duplicate vectors gives a 
(not necessarily harmonious) CO ordering. 
Thus, we can use Proposition~\ref{prop:soft set reconstruction} 
to find such a CO ordering $A$ and decide when 
it does not exist (in which case there must be no 
HCO multiordering of $L$ either). 
Now, Proposition~\ref{prop:extension} tells us that $A$ 
must be extendable into a harmonious CO multiordering, 
so it suffices to find such an extension.

Consider any pair of adjacent inharmonious vectors 
$x$ and $y$ in $A$. The coordinate-wise product $x*y$
is the vector described in the proof of 
Proposition~\ref{prop:extension}. If $x * y$ is in $L$,
insert a copy between $x$ and $y$. If not, then the reasoning
above implies that no HCO multiordering exists.
As the check for such an $x*y$ is constant time, 
our algorithm is still $O(n+k)$ time.
\end{proof}

We do not have a solution for the circular case. The main obstacle is that Proposition~\ref{prop:extension}, our main tool in this section, is not true for the circular case. 
\begin{ex}
Consider the HCCO multiordering (in fact, an ordering) at left.
\[ \begin{small} \left\{\begin{bmatrix} 1 \\ 0 \\ 0 \\ 0 \end{bmatrix},
\begin{bmatrix} 1 \\ 1 \\ 1 \\ 0 \end{bmatrix},
\begin{bmatrix} 0 \\ 1 \\ 0 \\ 0 \end{bmatrix},
\begin{bmatrix} 1 \\ 1 \\ 0 \\ 1 \end{bmatrix} \right\} \hspace{1cm} \left\{\begin{bmatrix} 1 \\ 0 \\ 0 \\ 0 \end{bmatrix},
\begin{bmatrix} 0 \\ 1 \\ 0 \\ 0 \end{bmatrix},
\begin{bmatrix} 1 \\ 1 \\ 1 \\ 0 \end{bmatrix},
\begin{bmatrix} 1 \\ 1 \\ 0 \\ 1 \end{bmatrix}\right\}\end{small}\]
It has a CCO ordering, at right, where the first two vectors are inharmonious. However, it is impossible to insert other vectors between the first two to create an HCCO multiordering; in fact, the insertion of any other vector breaks the CCO property. Thus, the reconstruction must use a different strategy.
\end{ex}

\subsection{An Example}
\label{sec:pctree}

The reader may be unfamiliar with the algorithm of PC-tree construction,
so we illustrate a non-trivial example here to supplement our other references:

Suppose we are given the following code: $\mathcal{C} = \{1100,1000,0100,
0000,0001,0110\}$ or, more succinctly, $\{12,1,2,0,4,23\}$. We carry out
the PC algorithm in Figure~\ref{fig:pctree}, which should be read
from left to right. The output is the ordering
in the bottom right; note that the small black (known as ``P'') nodes can be permuted
arbitrarily, whereas the white (``C'') nodes can only be
shifted with circular permutations.

\begin{figure}[!h]
\hspace{-1cm}
\begin{longtable}{cccc}
\includegraphics[scale=.5]{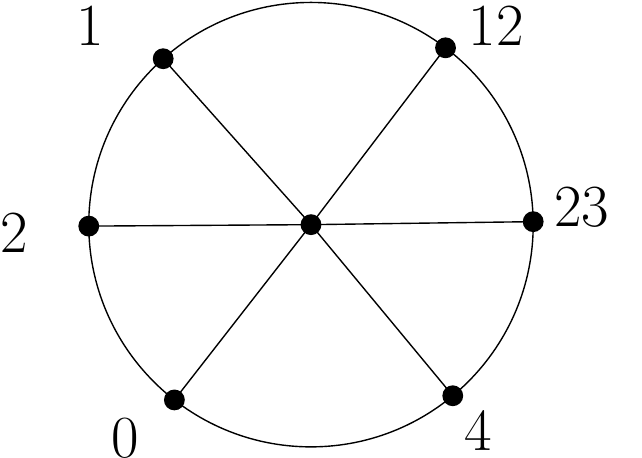} &
\includegraphics[scale=.45]{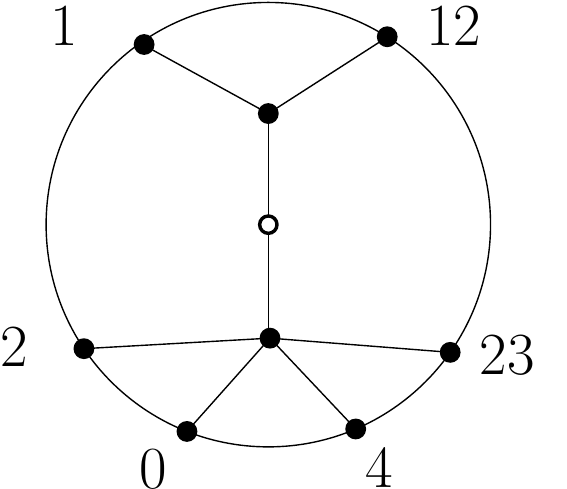} &
\includegraphics[scale=.45]{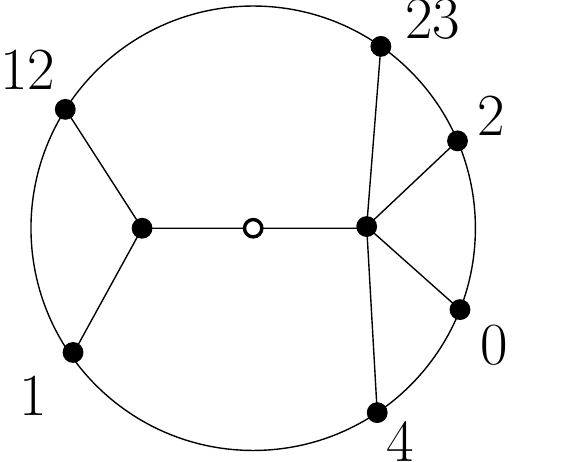} &
\includegraphics[scale=.45]{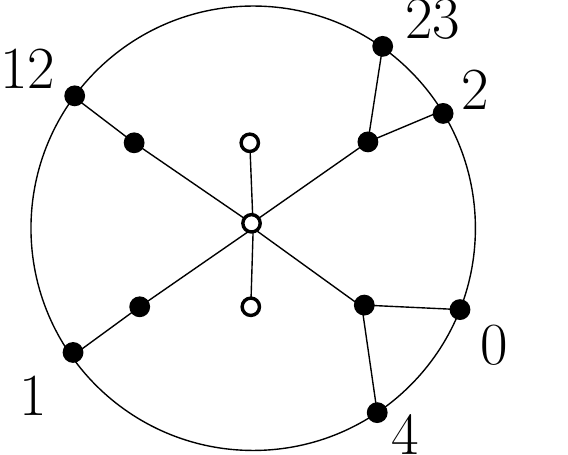} \\
\includegraphics[scale=.45]{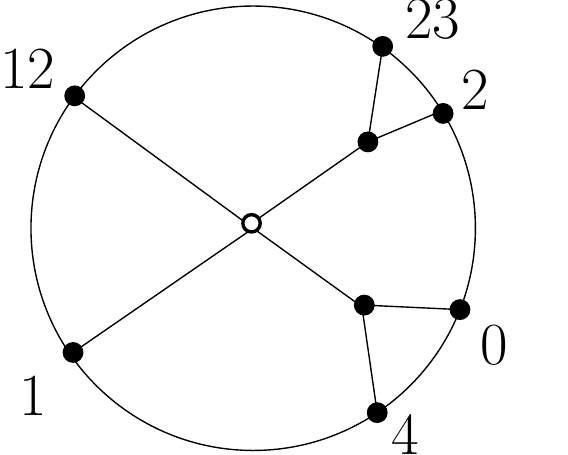} &
\includegraphics[scale=.45]{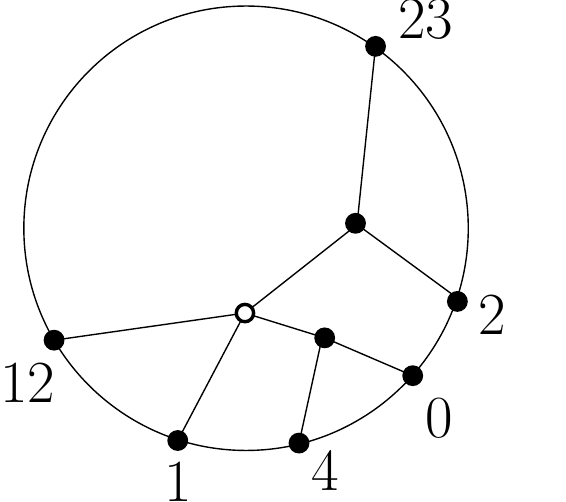} &
\includegraphics[scale=.45]{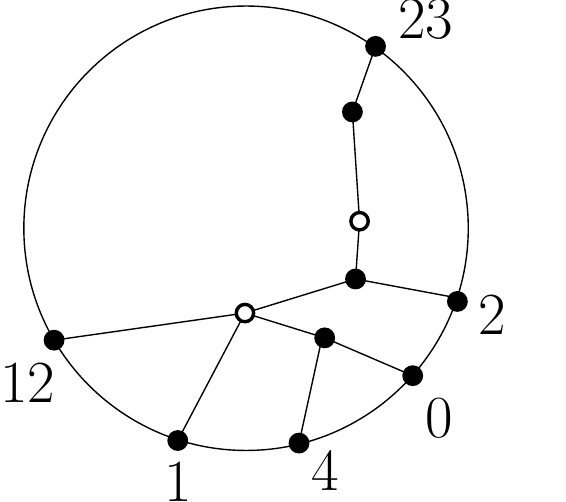} &
\includegraphics[scale=.45]{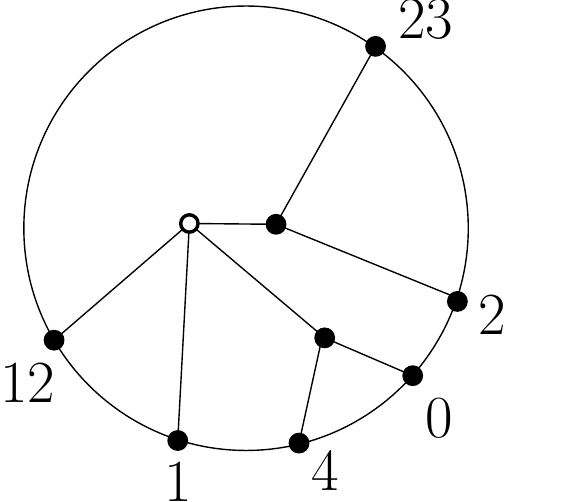} \\
\includegraphics[scale=.45]{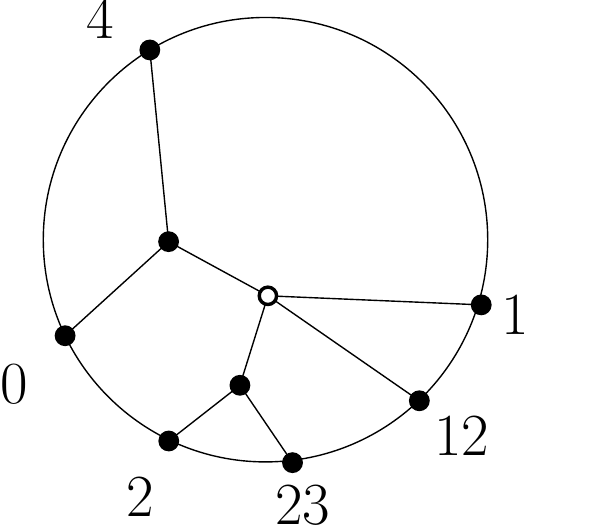} &
\includegraphics[scale=.45]{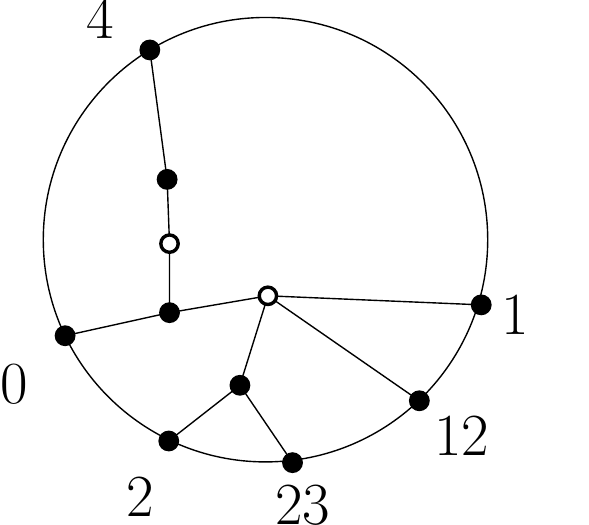} &
\includegraphics[scale=.45]{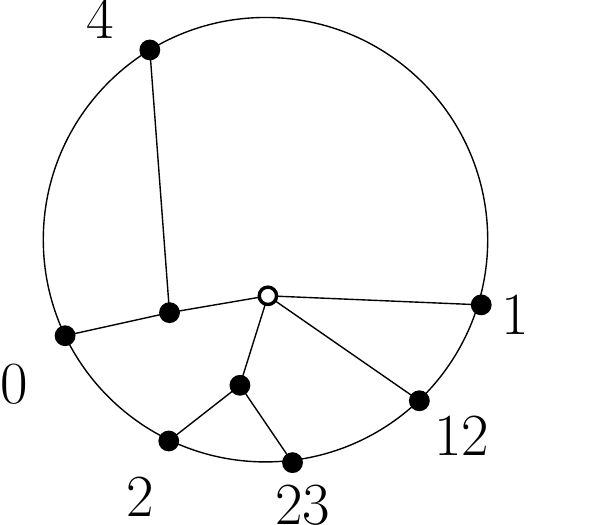} &
\includegraphics[scale=.45]{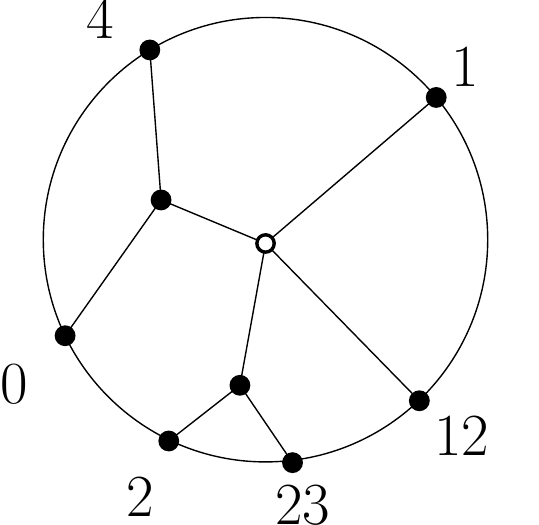} 
\end{longtable}

\label{fig:pctree}
\caption{Intermediate Stages of the PC Algorithm for $\mathcal{C}$}
\end{figure}

\pagebreak

Since the zero word appears in the code, any 
matrix coming from this PC tree is CO and CCO.
This handles the sensor-sparse case for both cases. 
To determine the sensor-dense case,
we fix a matrix and check for inharmonious pairs. Starting
from $4$ and moving clockwise, we obtain the matrix below left.
The only inharmonious pair is the first two columns. Since
their coordinate-wise product, the all-$0$'s word, is in the code, we can add it inbetween the two columns to obtain the HCO and HCCO matrix below right.

\[ \begin{bmatrix}
0 & 1 & 1 & 0 & 0 & 0 \\
0 & 0 & 1 & 1 & 1 & 0 \\
0 & 0 & 0 & 1 & 0 & 0 \\
1 & 0 & 0 & 0 & 0 & 0 \\
\end{bmatrix}\hspace{1cm}\longrightarrow \hspace{1cm} \begin{bmatrix}
0 & 0 & 1 & 1 & 0 & 0 & 0 \\
0 & 0 & 0 & 1 & 1 & 1 & 0 \\
0 & 0 & 0 & 0 & 1 & 0 & 0 \\
1 & 0 & 0 & 0 & 0 & 0 & 0 \\
\end{bmatrix}.\]


\section{The Multiset Reconstruction Problem}
\label{sec:multiset-reconstruction}

Recall that the multiset reconstruction problem is the following: 
given a multiset of $n$ codewords $L$, 
decide if there exists a CO (resp. HCO, CCO, HCCO)
 matrix (with exactly $n$ columns) 
such that the multiset of columns is $L$. 
As before, the sensor-sparse regime 
(for both the linear and circular versions) 
is much easier, but similar ideas can be used to prove the
linear sensor-dense case.

\begin{lem} The following holds for a CO  matrix $M$:
\begin{enumerate}[1.]
\item Removal of any column preserves CO.
\item Insertion of a duplicate of some column $v$ adjacent to $v$ preserves CO.
\item Insertion of a duplicate of some column $v$ adjacent to $v$ preserves HCO.
\end{enumerate}
\label{lem:multiset1}
\end{lem}

\begin{proof}
Given a row with a single interval of contiguous $1$'s, removing a $1$ in the row creates another (possibly empty) interval of $1$'s, and removing a $0$ outside of the interval does not change the interval. Thus, removal of any column preserves CO.

Now, suppose we duplicate a column next to itself. For each row, we originally had a single interval of contiguous $1$'s. We have either added a $1$ next to some $1$ in this interval or added a $0$ outside the interval where there was already a $0$. Thus, this operation preserves $CO$.

Finally, for sake of contradiction suppose the inserted column $v'$ creates an inharmonious pair while the original matrix is HCO. Without loss of generality $(v',w)$ is inharmonious, and $v$ is to the left of $v'$. Then removing $v'$ creates the inharmonious pair $(v, w)$ in the original matrix, a contradiction.
\end{proof}

\begin{lem}
Suppose a column $v$, appearing with multiplicity $m$ in an HCO matrix
$A$, has the property
that the removal of any instance introduces an inharmonious pair.
Then any HCO matrix with the same underlying set has $v$ appearing with
multiplicity at least $m$.
\label{lem:multiset2}
\end{lem}

\begin{proof} Take an instance of $v$ with this property, 
and let $(x,y)$ be an inharmonious pair. The column $v$ is either bound
between $x$ and $y$ or it is the coordinate-wise product $x*y$. 
If it is bound between $x$ and $y$, then by Lemma~\ref{lem:row selection},
so must any other instance of $v$; since removal would result in an
inharmonious pair, this means that $m = 1$, and the lemma is proven.

Therefore, the list of locations of $v$ is between pairs $(x_i, y_i)$ for
$i = 1,\ldots, m$ such that $v = x_i*y_i$ for all $i$. 
By $CO$, the coordinates of $v$ are $1$ precisely 
on the rows where every $x_i$ and $y_i$ is $1$.
Let $\tilde{A}$ be the submatrix of $A$  whose columns are greater 
than or equal to $v$, and whose rows are precisely those where $v$ is $0$.
Intuitively, this means resetting the ambient space as $\tilde{X} = 
\bigcap_{j: v_j = 1} I_j$, and then considering the interval arrangement
$\{I_k\}_{k: v_k = 0}$.

Using Definition \ref{defn:connected}, we claim
that nonzero blocks of $\tilde{A}$  correspond to $R_{v = 0}$-connected
components of the graph. Since this graph is constructed from
the underlying set (not the multiset), it is invariant under reordering.
Because these components are nonzero, we will need exactly one
zero column between adjacent blocks to preserve harmoniousness. 
Thus the minimal multiplicity is fixed under any ordering.
\end{proof}

\begin{prop}
\label{prop:soft multiset reconstruction}
Suppose we are in (either the linear or the circular case of) 
the sensor-sparse regime. Given a multiset of $n$ length-$k$ 
vectors $L$, there is a $O(n+k)$ algorithm that solves 
the multiset reconstruction problem. Specifically, the 
algorithm decides if there exists a CO (or CCO) matrix with
 $n$ columns in bijection with $L$ and constructs one if it exists. 
\end{prop}
\begin{proof}
The main observation is that if $L'$ is the underlying set 
(i.e. removing all duplicates) of $L$, then there is a CO 
matrix with the column multiset equal to $L'$ if and only 
if there is such a matrix with the column multiset equal to $L$.
Using the results of Lemma \ref{lem:multiset1}, it is enough
to consider the question for $L'$. If a CO matrix exists for $L'$
then padding the matrix appropriately gives a realization of $L$.

Thus, the multiset reconstruction problem for $L$ is equivalent 
to the set reconstruction problem for $L'$ and can be resolved 
with Proposition~\ref{prop:soft set reconstruction}. 
The logic used in this proof flows identically for the 
CCO case as the CO case.
\end{proof}

Similar logic allows us to address the sensor-dense regime.

\begin{prop}
\label{prop:hard multiset reconstruction}
Suppose we are in the linear case of the sensor-dense regime. Given a multiset of $n$ length-$k$ vectors $L$, there is an $O(n+k)$ algorithm that solves the multiset reconstruction problem. Specifically, the algorithm decides if there exists an HCO matrix with $n$ columns 
in bijection with $L$ and constructs one if it exists. 
\end{prop}
\begin{proof}As before, we construct an HCO matrix on the underlying set
of $L$ using the algorithm of Theorem 
\ref{thm:1d reconstruction CO-harmonious}.
Suppose the resulting multiset of columns is $L'$, which is not in
bijection with $L$. Then $L'$ has, for some vector $v$,
either more copies or fewer copies than $L$. We claim that for each
such $v$ we can change $M$ to retain an HCO matrix such that the number
of appearances of $v$ in $M$ (and thus in $L'$) matches that of $L$.

If $L'$ has fewer copies of $v$ than $L$, Lemma \ref{lem:multiset1}
allows us to add in duplicates next to other instances of $v$.
If $L'$ has more copies of $v$ than $L$, we remove any copy that
does not leave an innharmonious pair. By Lemma~\ref{lem:multiset2},
this number will be the minimum possible achieved by an HCO matrix. 
So, if this number is still greater than desired, we can decide
the multiset as being unrealizable.
\end{proof}

As before, the conjunction of the CCO property and harmoniousness
seems mysterious, and we do not have a solution to the multiset 
reconstruction problem for the circular case of the sensor-dense regime.

\section{The Enumeration Problem}

Suppose we have $k$ sets and $n$ sensors in one of our regimes. We can count several related objects, such as possible codes or possible sensor matrices. For us, counting the following seems to be the most natural: call a set of $k$ \textbf{distinct} length-$n$  vectors a \emph{discrete interval set} if it is the set of rows for for some $k \times n$ sensor matrix in the corresponding regime. Any sensor matrix can be constructed by taking such a set, freely duplicating rows, and permuting the rows in some order. For each regime, let the number of discrete interval sets of $k$ vectors in $\FF_2^n$ be $c_{n,k}$. Our problem then becomes finding $c_{n,k}$ in some nice form, be it closed form, generating functions, etc.  

\subsection{Linear Enumeration}
\label{sec:linear-enumeration}

The sensor-sparse regime is again very simple:

\begin{prop}
\label{prop:soft enumeration}
For the linear case of the sensor-sparse regime, we can enumerate the discrete interval sets via $c_{n,k} = {{n + 1 \choose 2} \choose k}$. For the circular case, $c_{n,k} = {n^2-n+1 \choose k}$ for $n \geq 2$.
\end{prop}
\begin{proof}
For the linear case, we just need each of the $k$ rows 
to be a nontrivial CO row vector, as having a row in 
the discrete interval set is independent of having any other row. 
Consider the $(n+1)$ ``gaps'' between the $n$ positions 
(counting the two ``gaps'' on the boundary). 
There is a bijection between nontrivial CO rows and 
the choice of two distinct such gaps -- put $1$'s 
in all the positions between the two gaps and $0$'s 
elsewhere. This gives ${n + 1 \choose 2}$ possible 
row vectors, of which we can arbitrarily select $k$. 

For the circular case, we can describe a consecutive 
arc of $1$'s that does not cover the whole row by 
picking any of the $n$ elements as the first $1$, 
then having $1$'s to its right (wrapping around if necessary), 
stopping after we have a number of $1$'s between $1$ and $(n-1)$. 
This gives $n(n-1)$ such choices, for which we need to add 
$1$ (corresponding to having the entire row be $1$, 
which we forbade earlier to avoid overcounting). 
This gives a total of $n^2-n+1$ row vectors. 
\end{proof}

We now go to the sensor-dense regime. 

\begin{thm}
\label{thm:harmonious-gf}
For the linear case of the sensor-dense regime, we can enumerate 
discrete interval sets by computing the following generating 
function for the numbers $c_{n,k}$:
\[
f(x, y) = \sum_n \sum_k c_{n,k} x^ny^k = \sum_{m=0}^\infty \frac{x^m}{(1-a_1x)(1-a_2x)\cdots(1-a_{m+1}x)},
\]
where $a_i = (1+y)^i - 1$.
\end{thm}

\begin{proof}
Consider a discrete interval set $R$. Note that each row $r \in R$, since it is nonzero, must have a first occurrence and a last occurrence of $1$. Index the $n$ coordinates of $r$ by $[n]$. Define $f(r)$ to be the last index where $r$ has value $1$ and define $g(r)$ to be $1$ less than the first index where $r$ has value $1$ (it may also be useful to think of $r$ as a CCO row and think of $g(r)$ as the last index where $r$ has value $0$, wrapping around circularly if necessary). For example, the row $r = 001110$ has $f(r) = 5$ and $g(r) = 3-1 = 2$. Note that our conditions for $R$ being a discrete interval set (i.e. any, equivalently all, of the matrices with rows coming from $R$ is HCO) is satisfied precisely when there is no pair of rows $r_1$ and $r_2$ in $R$ such that $g(r_1) = f(r_2)$.

Now, for $S \subset [n]$, let $E_S$ be the set of discrete interval sets where the range of $f(r)$ over the rows is exactly $S$. This creates a partition of all possible discrete interval sets. Suppose we have $S = \{i_1, \ldots, i_m\}, i_1 < i_2 < \cdots < i_m \leq n$. We define auxilary variables $i_1' \leq i_2' \leq \cdots \leq i'_m \leq n - m + 1$, where $i'_j = i_j - j + 1$. To count $E_S$, recall that all rows $r$ must have $f(r) \in S$. We claim that the contribution of $E_S$ to the $[x^n]$ coefficient of $f(x,y)$ is exactly
\[
((1+y)^{i_1'} - 1) ((1+y)^{i_2'} - 1) ((1+y)^{i_3'} - 1)\cdots ((1+y)^{i_m'} - 1).
\]
To see this, first consider the rows $r$ with $f(r) = i_1$. This provides a (multiplicative) contribution of $(1+y)^{i_1'} - 1$. This is because any of the smaller coordinates can either exist as $g(r)$ for some row in $R$ (in which case we pick up a power of $y$) or not, giving $(1+y)^{i_1'}$, from which we must subtract $1$ since we need at least one row with $f(r) = i_1$. Now consider the rows $r$ with $f(r) = i_2$; any of the smaller coordinates except $i_1$ can be used as $g(r)$, which gives $(1+y)^{i_2 - 1} - 1 = (1+y)^{i_2'} - 1$ choices by similar logic. Repeating this argument gives our expression above. If we denote $a_i = 2^i - 1$, then we can get the $[x^n]$ coefficient in $f(x,y)$ by summing over all terms $a_{i_1'}a_{i_2'}\cdots a_{i_m'}$, $i_1' \leq \cdots \leq i_m'$ for all $n >= i_m' + m - 1$. Hence, we obtain
\[
f(x, y) = \sum_m \sum_{i_1' \leq i_2' \leq \cdots \leq i_m'} a_{i_1'}a_{i_2'}\cdots a_{i_m'} \frac{x^{i_m' + m - 1}}{1-x}. 
\]

\begin{figure}
\begin{minipage}{.5\textwidth}
\flushright $\begin{matrix}
& 1 & 2 & 3 & 4 & 5 & 6 & 7 & 8 \\ \hline
& 1 & 0 & 0 & 0 & 0 & 0 & 0 & 0 \\
\bigcirc & & \bigcirc    & \bigcirc  & 1 & 0 & 0 & 0 & 0 \\
\bigcirc & & \bigcirc    & \bigcirc  & & \bigcirc  & \bigcirc & \bigcirc & 1 \\
\hline 
\end{matrix}$
\end{minipage} \hspace{5mm}
\begin{minipage}{.35\textwidth}
$S = \{1,4,8\}$ \\
$\bigcirc \longrightarrow$ Possible value of $g(r)$. \\[3mm]
Contribution of $E_S$: \\
$y \cdot ((1 + y)^3 - 1) \cdot ((1 + y)^6 - 1)$
\end{minipage}
\caption{Illustrative example for Proof of Theorem~\ref{thm:harmonious-gf}.}
\end{figure}
We now rethink the $[x^n]$ coefficient in $f(x, y)$. 
Our expression tells us that $[x^n]$ should get a 
contribution of $1$ (from the null set), any monomial 
of degree $1$ with index at most $n$ (from sets of size $1$), 
any monomial of degree $2$ with highest index at most $n - 1$, 
$\ldots$, the monomials of degree $n-1$ with highest index at most $2$, and finally $a_1^n$. The function in the statement of the theorem can be rewritten as
\[
\frac{1}{1-a_1x} + \frac{x}{(1-a_1x)(1-a_2x)} + \frac{x^2}{(1-a_1x)(1-a_2x)(1-a_3x)} + \cdots.
\]
Here, the first term picks up $a_1^n$, 
the second term picks up any monomial 
(again of the $a_i$'s) of degree $n-1$ with only $a_1$ 
and $a_2$, $\ldots$, monomials of degree $1$ with only 
$a_1$ through $a_n$, and finally $1$. 
These are exactly the previously-determined contributions 
of the coefficient $[x^n]$ in $f(x, y)$ in reverse order. 
Thus, these generating functions are equal. 
\end{proof}

Now suppose we substitute $y = 1$ (which gives $a_i = 2^i - 1$). 
We now obtain the function $\sum c_n x^n$, where $c_n = \sum_k c_{n,k}$. 
Here, $c_n$ counts the number of discrete interval sets of length-$n$ rows. 
As an example, there are $3$ possible nonzero rows of length $2$: 
$\{a = 01, b = 10, c = 11\}$. Any of the $2^3 = 8$ subsets of these 
rows form a (possibly trivial) harmonious CO matrix, 
except for $\{a,b\}$ and $\{a, b, c\}$ since $a$ and 
$b$ force inharmoniousness, so $c_2 = 6$. The first few elements are 
\[
1, 2, 6, 26, 158, 1330, \ldots
\]
which happens to be in OEIS \cite{oeis} (see \url{https://oeis.org/A135922}); it is the inverse binomial transform of $d_n = \sum_{k=0}^n {n \choose k}_2$, the number of all linear subspaces of $\FF_2^{n}$, but with $n$ shifted by $1$. This curious link to linear subspaces leads to the following idea: for any subspace $S \subset \FF_2^{n}$, define its \emph{support} to be the subset of coordinates in $[n]$ which is nonzero for at least one vector in $S$. A strange corollary of Theorem~\ref{thm:harmonious-gf} follows:
\begin{cor}
\label{cor:strange-bijection}
The number of discrete interval sets of length $n$ is equal to the number of subspaces of $\FF_2^{n+1}$ with support equal to $[n+1]$.
\end{cor}
\begin{proof}
We know our number $c_n$ is equal to $\sum_k (-1)^k {n+1 \choose k}\sum_{i = 0}^k {k \choose i}_2$. One can imagine that this sum represents the following formal sum over subspaces of $\FF_2^{n+1}$: we first pick $k$ coordinates for an embedding of $\FF_2^k$ inside $\FF_2^{n+1}$, find a subspace of any dimension $i \leq k$ of that copy of $\FF_2^k$, which induces a subspace of $\FF_2^{n+1}$ (by assigning $0$'s to any coordinates outside of those $k$ used), and assigning the coefficient $(-1)^k$ formally to this subspace. Equivalently, we could have, for any subspace $S$ with support set $K \subset [n+1]$, summed $(-1)^{|K'|}$ for all supersets $K' \supset K$. This formal sum is always $0$ unless $K' = [n+1]$, in which case it contributes $1$. Thus, we obtain exactly the number of subspaces of $\FF_2^{n+1}$ with support equal to $[n+1]$.
\end{proof}

\begin{ex}
Recall that $c_2 = 6$. We may check that there are indeed $6$ subspaces of $F_2^3$ with support equal to $[3]= \{1,2,3\}$. These are realized by generators as $\langle 111 \rangle$, $\langle 100, 011 \rangle$, $\langle 010, 101 \rangle$, $\langle 001, 110 \rangle$, $\langle 100, 010, 001 \rangle$, and $\langle 110, 011 \rangle = \langle 110, 101 \rangle = \langle 011, 101 \rangle$.  This is a curious coincidence! Besides the strange shift by $1$, the two objects seem very different: discrete interval sets are defined by pairwise local restrictions (being harmonious) and linear subspaces are defined by set-wise conditions (linear independence). A more direct bijection may be enlightening here. 
\end{ex}

Enumerating codes directly seems to be harder than 
enumerating discrete interval sets. One strategy may 
be to encode the problem algebraically, such as in 
the direction of work of Curto et al. \cite{curto1,Curtoetal}, 
and then apply 
brute-force computational search. 
It would be helpful to have some bounds, 
e.g. a minimal number of columns necessary to realize all 
codes on $n$ intervals, to use this strategy.

\subsection{Circular Enumeration}
\label{sec:circular-enumeration}

Finally, we enumerate discrete interval sets for the circular case. Interestingly, while the circular case was harder for the set reconstruction problem, it is the easier case for the enumeration problem due to its added symmetry.

\begin{thm}
\label{thm:harmonious-cco-gf}
For the circular case of the sensor-dense regime, we can enumerate discrete interval sets by computing the following generating function for the numbers $c_{n,k}$:
\[
f(x, y) = \sum_n \sum_k c_{n,k} x^ny^k = 1 + \sum_{m=1}^\infty \frac{2x^m}{(1-a_mx)^{m+1}}.
\]
Here, $a_i = (1+y)^i - 1$.
\end{thm}

\begin{proof}
For this case, we exclude the consideration of the all $1$'s row since it behaves differently from other rows. This is not a problem as having the all $1$'s row is harmonious with all other rows, so it just contributes a factor of $2$ at the end for $n \geq 1$ (as for $n = 0$ there is no all $1$'s row). Let $h(x,y)$ be the generating function for these restricted discrete interval sets where the all $1$'s row is not allowed. We can then obtain $f(x,y) = 2h(x,y) - 1$ by our argument above.

Consider a discrete interval set $R$. As before, define $f(r)$ to be the last index where $r$ has value $1$ and define $g(r)$ to be the last index where $r$ has value $0$, allowing wraparound. For example, the row $r = 110001$ has $f(r) = 2$ and $g(r) = 5$. Note that $f$ and $g$ are not well-defined when $r$ is the all $1$'s vector, which is why we excluded it. As before, our conditions for $R$ being a discrete interval set is satisfied precisely when there is no pair of rows $r_1$ and $r_2$ in $R$ such that $g(r_1) = f(r_2)$.

Again, for $S \subset [n]$, let $E_S$ be the set of discrete interval sets where the range of $f(r)$ over the rows is exactly $S$. In this case, counting $|E_S|$ is much easier. For each choice of $S = \{i_1, \ldots, i_m\}$, there are $(n-m)$ indices available as the codomain for $g(r)$. The contribution to the $[x^n]$ coefficient of $f'(x,y)$ from $E_S$ is just $((1+y)^{n-m} - 1)^m$, because for any of the $m$ values in the range of $f(r)$, the other $(n-m)$ indices either exist as $g(r)$ for some row or not, and again we subtract $1$ since we need at least one such row.


As before, if we denote $a_i = 2^i - 1$, then we can get the $[x^n]$ coefficient in $h(x,y)$ by summing this contribution for $E_S$ over all $S$. There are exactly ${n \choose m}$ possible sets $E_S$ for every $m$, so we obtain
\[
\begin{array}{cccccccc}
h(x, y) & = & \displaystyle \sum_{n=0} \displaystyle \sum_{m=0} 
{n \choose m} a_{n-m}^m x^n & = & \displaystyle \sum_{n=0} \displaystyle
 \sum_{m=0} {n \choose m} a_m^{n-m} x^n & = & \displaystyle \sum_{k=0} 
\frac{1}{a_m^m} \displaystyle\sum_{n = 0} {n \choose m} a_m^n x^n \\
& = & \displaystyle \sum_{k=0} \frac{1}{a_m^m} \frac{(a_mx)^m}{(1-a_mx)^{m+1}} &
= & \displaystyle \sum_{k=0} \frac{x^m}{(1-a_mx)^{m+1}} \: .
\end{array}
\]

Finally, multiplying by $2$ and subtracting $1$ obtains $f(x,y)$ when we allow the all $1$'s row.
\end{proof}

We can again substitute $y = 1$ to obtain the generating function for $c_n = \sum_k c_{n,k}$, the number of discrete interval sets of length-$n$ rows. 

\begin{ex}
As an example, we compute $c_3$. Since the all $1$'s row just gives a factor of $2$, we can again exclude it. The vectors we can use are $110, 101, 011, 001, 010,$ and $100$. Note that if we use any of the rows with two $1$'s, the other such rows must be excluded. Suppose we have one such vector (w.l.o.g. $110$) then we cannot also have $001$, but we can have at most one of $100$ and $010$, giving $3$ choices. By symmetry, this gives $9$ discrete interval sets with at least one vector with two $1's$. If we have no such vectors in our discrete interval set, then we can have at most one of $100$, $010$, and $001$, for $4$ more discrete interval sets for a total of $13$. Multiplying by $2$ gives $26$. It is very strange that $c_3$ for the linear case is also $26$, but the two sets look nothing alike!
\end{ex}

The first few $c_n$ are 
\[
1, 2, 6, 26, 174, 1684, \ldots
\]
While this sequence is not in the OEIS, the sequence for $n > 1$ with all elements halved is \cite{oeis} (see \url{https://oeis.org/A001831}). This sequence counts labeled graded $(3)$-avoiding posets, or alternatively, $n \times n$ square $0$-$1$ matrices which square to the zero matrix. Again, a direct bijection would be enlightening.

\section{Conclusion}
\label{sec:conclusion}

We addressed the problems of set/multiset reconstruction and enumeration for
$1$-dimensional (convex neural) codes, under different (sensor-sparse or sensor-dense) topological assumptions. The theory of consecutive-ones matrices is very helpful for the former problem and the latter problem can be attacked with generating functions. The added requirement of harmoniousness makes the circular case more difficult than the linear case. Along the way, we have provided an algorithm that 
determines whether a code is realizable in dimension $1$ and one that provides a certificate for rejection.

\vspace{5mm}

\centerline{
\begin{tabular}{c|cccc}
	& Linear Sparse 	& Circular Sparse 	& Linear Dense 	& Circular Dense 	\\ \hline
Reconstruction	& $\surd$	& $\surd$	& $\surd$	& 		\\
Enumeration*	& $\surd$	& $\surd$	& $\surd$	&$\surd$	\\
\end{tabular}
}

\vspace{1mm}

\subsection{Higher Dimensions}

There are obvious generalizations of the ideas in
this paper to codes of higher dimension, especially considering that higher-dimensional convex sets project to convex sets in lower dimension. However, it becomes trickier to consider different topological regimes. For example, the formulation of the reconstruction problem in \cite{Curtoetal} assumes the \textbf{open} convex sets, but we may also want to consider closed convex sets, arbitrary convex sets, etc. We start such a discussion with Appendix~\ref{sec:topology}, but there is clearly much more to be done. 


Another nontrivial aspect of this project in higher dimensions is the correct idea of convexity. In dimension $1$, the idea of a set of codewords translating into a discrete union of convex intervals has a straightforward interpretation. However, in higher dimensions, there are various definitions of discrete (or digital) convexity. Webster \cite{webster} believed that the
right setting for digital convexity is the {\em abstract cell complex} (ACC),
first introduced by Kovalevsky \cite{Kovalevsky}. The ACC
includes cubes of dimension $n$ and then all lower-dimensional
faces of those cubes. In this setup, unlike other proposed
schemes for digital convexity, straightforward analogues of 
the classical results Radon's Theorem, Helly's Theorem,
and Caratheodory's Theorem all hold.
Unfortunately, this definition relies on the use of closed hyperplanes.
The analog of an open set would have any face automatically
include all neighboring higher-dimensional faces, which would make all these theorems fail. Indeed, the only open convex sets under this definition whose rotations are a Helly family is the family of rectangles. Some other versions of convexity are summarized in \cite{eckhardt}:
\begin{enumerate} \setlength{\itemsep}{-2mm}
\item {\em MP-convexity}: Two lattice points $x,y$ in a convex set $S$
means that any lattice point on their line segment is also in $S$.
\item {\em H-convexity}: A convex set $S \subset \ZZ^d$ must contain all lattice points
in its convex hull as a subset of $\RR^d$.
\item {\em D-convexity} and {\em DH-convexity}: Different versions
of digital lines are used to modify MP-convexity.
\end{enumerate}

For our purposes, coming from a sampled convex set in
a real Euclidean space, H-convexity is the most attractive.
However, due to the failure of Radon, Caratheodory and Helly,
proofs may be difficult for H-convexity.

\subsection{Open Problems}

Some natural combinatorial questions remain:
\begin{itemize} \setlength{\itemsep}{-2mm}
\item Does there exist a linear-time algorithm for determining
whether a code is harmonious circular consecutive-ones?
 Is there a rejection criterion?
\item How to enumerate (generating function?) codes as opposed to the discrete interval sets?
\item How many intervals are necessary to realize every code on $n$ neurons with $k$ codewords?
\item Are there good explanations / bijections to the other combinatorial objects found on the OEIS in Sections~\ref{sec:linear-enumeration} and \ref{sec:circular-enumeration}?
\item Can we bound the size of the matrix necessary to realize all codes on $n$ neurons with $k$ codewords? This seems especially important for applications going into higher dimensions. For example, we may then be able to do sampling (or even search) on the space of codes.
\end{itemize}

\vspace{-1mm}

\subsection*{Acknowledgements}
This project was initiated at a 2014 AMS Mathematics Research Community, 
``Algebraic and Geometric Methods in Applied Discrete Mathematics,''
supported by NSF DMS-1321794. Early stages of the research were
enriched by insights from Carina Curto, Elizabeth Gross, Jack Jeffries, Katie Morrison, Mohamed Omar, Anne Shiu, and Nora Youngs. We thank Carina Curto for reviewing early drafts of the work.
ZR was supported in part by a Simons Foundation Math+X research grant for the
late stages of this research.

\vspace{-1mm}

\bibliographystyle{alpha}
{\small
\bibliography{RZsources}}

\appendix

\section{Topological Considerations}
\label{sec:topology}

In this paper, we worked with \textbf{open} convex intervals. Two other natural choices are:
\begin{itemize}
\item arrangements of convex closed intervals;
\item arrangements of convex arbitrary intervals (open, closed, or half-open).
\end{itemize}
In this Appendix, we discuss the relationship between these questions.

\begin{prop}
\label{prop:allsame}
The set of codes realizable by any of the three types of constraints (open convex, closed convex, arbitrary convex) in the sensor-sparse regime are identical.
\end{prop}
\begin{proof}
  Given an arrangement under any of the three types of constraints above, consider the discrete set of sensors $S$. For an interval with an open endpoint in $S$, we may slightly shorten the interval such that the endpoint is no longer in $S$; the code is invariant as $S$ did not detect the interval to begin with. Similarly, for an interval with a closed endpoint in $S$, we may slightly enlarge the interval so that the endpoint is no longer in $S$. Therefore, for each code we may assume our arrangement has no endpoints in $S$. As changing the open/closedness of points in $S$ does not affect our code, we can then realize our code as an arrangement in any of the $3$ types of constraints above. This shows all the constraints are identical.
\end{proof}
 
\begin{prop}
\label{prop:open=closed}
The set of codes realizable by open intervals is identical to the set of codes realizable by closed intervals in the sensor-dense regime.
\end{prop}
\begin{proof}
Suppose we have an arrangement of open intervals $A$ with a corresponding code $\mc{C}$. It is possible to pick some $\epsilon > 0$ small enough such that replacing all intervals $(a,b)$ by $(a+\epsilon, b - \epsilon)$ keeps the code intact while ensuring that no right endpoint of any interval equals the left endpoint of any interval. Replacing every intervals $(a,b)$ by its closure $[a,b]$ now creates an arrangement of closed convex intervals where the only possible changes to the codewords occur at endpoints of intervals. As no right endpoint of any interval equals the left endpoint of any interval, we do not gain or lose any codewords by replacing each interval with its closure. To get from an arrangement of closed intervals to one of open intervals, we can reverse this process; first we enlarge all the closed intervals, then replace them by open intervals. The logic works similarly.
\end{proof}

The above proposition is subtle; one might conjecture 
that the same would hold in higher dimension, but a 
counterexample was identified by \cite{shiu}; the code is 
realizable by closed sets in dimension $2$, but is not 
realizable in \textbf{any} dimension for open sets. 
A similar example was identified in \cite{itskov2016}: 
the code is realizable for open sets 
in dimension $2$, but not for closed convex sets in any dimension.
Thus, this Proposition is clearly confined to dimension $1$.

\begin{prop}
The set of codes realizable by arbitrary intervals in the sensor-dense regime is identical to the set of codes realizable by open (equiv. closed or arbitrary) intervals in the \textbf{sensor-sparse} regime.
\label{prop:arbitrary}
\end{prop}
\begin{proof}
Suppose we have an arrangement of arbitrary intervals in the sensor-dense regime giving some code $\mc{C}$. Restricting to any finite set of representative sensors that detect $\mc{C}$ gives the same  arrangement with the same code $\mc{C}$ in the sensor-sparse regime. 

Now suppose we have an arrangement of open intervals with a 
finite set of sensors $S$ detecting some code $\mc{C}$. 
We can replace it by an arrangement of intervals of form 
$[a,b)$ with both $a, b \in S$ by rounding up both endpoints 
of each interval to the nearest sensor in $S$ and then making 
the left endpoint closed and the right endpoint open. This 
does not affect which intervals each sensor detects, 
so the code is preserved. The boundary cases here are: \begin{enumerate}
\item Both endpoints of an interval are between 
two adjacent points in $S$. In this case, just discard the interval.
\item In the circular case where both endpoints are 
between two adjacent points in $S$ in opposite order, 
in which case we replace the interval with the entire ambient space ($\RR$ or $S^1$). 
\end{enumerate}
In both cases the remainder of the proof goes through. An arrangement of this form has the property that any point between two adjacent sensors $S_1$ and $S_2$ sees exactly what $S_1$ sees, which means the code of this new arrangement in the sensor-dense regime does not have any new codewords, and must in fact be equal to $\mc{C}$.

\end{proof}
\begin{figure}[h!]
\includegraphics{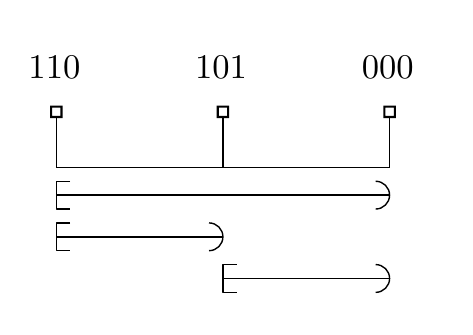}
\caption{Illustration of the Proof of Proposition~\ref{prop:arbitrary}}
\end{figure}
Among the six possibilities of sensor density and interval topology, only two distinct cases emerge: a sensor-sparse case that ``does not see'' topology (and happens to be equivalent to sensor-dense for arbitrary intervals) and a sensor-dense case where the intervals are either all open or all closed. Thus, we reduce to two cases by fixing our intervals to be open convex and just considering sensor density.

\end{document}